\documentclass[10.5pt]{article}
\usepackage{amssymb}
\usepackage{amsthm}
\usepackage{amsmath, mathrsfs}
\usepackage{amsmath}
\usepackage{amssymb}
\usepackage{amsthm}
 \newtheorem{thm}{Theorem}[section]
 \newtheorem{cor}[thm]{Corollary}
 \newtheorem{lem}[thm]{Lemma}
 \newtheorem{prop}[thm]{Proposition}
 \theoremstyle{definition}
 \newtheorem{defn}[thm]{Definition}
 \theoremstyle{remark}
 \newtheorem{rem}[thm]{Remark}
 \newtheorem*{ex}{Example}
 \numberwithin{equation}{section}

\newcommand{\norm}[1]{\left\Vert#1\right\Vert}
\newcommand{\scal}[1]{\left<#1\right>}
\newcommand{\Hq}{\mathbb H}
\newcommand{\Sq}{\mathbb S}

\newcommand{\N}{\mathbb{N}}

\newcommand{\R}{\mathbb{R}}      
\newcommand{\Z}{\mathbb{Z}}
\newcommand{\C}{\mathbb{C}}

\newcommand{\CFH}{\mathcal{F}_{Slice}^\alpha(\Hq)}

\newcommand{\DS}{\mathcal{D}_{\alpha,S}}
\newcommand{\MS}{\mathcal{M}}
\newcommand{\Io}{\mathcal{I}}

\begin{document}

\title{The Cholewinski-Fock space in the Slice Hyperholomorphic Setting}

\author{
Kamal Diki \thanks{Marie Sklodowska-Curie fellow of the Istituto Nazionale di Alta Matematica}  \\
Dipartimento di Matematica\\
Politecnico di Milano\\
Via Bonardi 9\\
20133 Milano, Italy\\
kamal.diki@polimi.it}
\date{}
\maketitle

\begin{abstract}
Inspired from the Cholewinski approach see \cite{Cholewinski1984}, we investigate a family of Fock spaces in the quaternionic slice hyperholomorphic setting as well as some associated quaternionic linear operators. In a particular case, we reobtain the slice hyperholomorphic Fock space introduced and studied in \cite{AlpayColomboSabadini2014}.
\end{abstract}



 \section{Introduction}
 \label{s1}
In 1961 Bargmann introduced in his original paper \cite{Bargmann1961} a Hilbert space of entire functions on which the creation and annihilation operators, namely $$\displaystyle M_zf(z):=zf(z)\quad \text{and} \quad Df(z):=\frac{d}{dz}f(z)$$ are closed, densely defined operators that are adjoints of each other and satisfy the classical commutation rule $$\left[D,M_z\right]=\mathcal{I}$$ where $\left[.,.\right]$ and $\mathcal{I}$ are respectively the commutator and the identity operator. In the literature, this space is known also as Fock or Segal-Bargmann space. It consists of entire functions that are square integrable on the complex plane with respect to the normalized Gaussian measure. It turns out that the standard Schr\"odinger Hilbert space on the real line is unitary equivalent to the Fock space via the so-called Segal-Bargmann transform. A few years later, in \cite{Cholewinski1984} Cholewinski extended this construction by studying a Hilbert space of even entire functions weighted by a modified Bessel function of the third kind sometimes also called Macdonald function. His construction generalized the original one of Bargmann so that in a particular case the weight is exactly the classical normalized Gaussian measure. He also proved in \cite{Cholewinski1984} some commutator relations between the Schr\"odinger radial kinetic energy operator and the operator $M_{z^2}$. Then, in 2002 based on the approach of Cholewinski, Sifi and Soltani considered and studied in \cite{SifiSoltani2002} a Hilbert space of entire functions that are not necessarily even with a weight involving the Macdonald function. \\ \\ In 2007 a new theory extending the classical one of holomorphic functions in complex analysis to the quaternionic setting has appeared in \cite{GentiliStruppa07}. It was also extended to the so-called slice monogenic setting for Clifford algebras valued functions, see \cite{CSS2009}. This new theory was extensively developed and found several applications in the last years in different mathematical fields, for example in Schur analysis and quaternionic operator theory, see \cite{ACS2016,ColomboSabadiniStruppa2011}. Moreover, this new theory may be very useful to develop the formalism for quaternionic quantum mechanics, see \cite{MTS2017}. Recently, the topic of Segal-Bargmann-Fock space and associated integral transforms in this new quaternionic and slice monogenic setting attracted the attention of several mathematicians and physicists from different points of view, see \cite{AlpayColomboSabadini2014,CD2017,CSS2017,DMNQ2018,DG2017,KMNQ2016,MNQ2017,PSS2017}. The purpose of this paper is to continue this exploration following the approach of Cholewinski, Sifi and Soltani in order to present a study of a quaternionic Hilbert space of slice entire functions weighted by a modified Bessel function that we shall call the quaternionic slice hyperholomorphic Cholewinski-Fock space or the slice Cholewinski-Fock space for short. This will allow us to extend some results obtained in \cite{AlpayColomboSabadini2014,DG2017} on the slice hyperholomorphic Fock space and the quaternionic analogue of the Segal-Bargmann transform. Moreover, we study some specific quaternionic operators associated to the slice Cholewinski-Fock space. In a particular case, we show that the slice derivative and the quaternionic multiplication are adjoints of each other and satisfy the classical commutation rule on the slice Fock space introduced in \cite{AlpayColomboSabadini2014}. \\ \\ The paper has the following structure: in the next section we briefly collect some basic facts about slice hyperholomorphic functions on quaternions and review some useful properties of the Macdonald function as it will be needed in this paper. In section 3, we define the slice Cholewinski-Fock space and we introduce an orthonormal basis. Moreover, we show that it is a quaternionic reproducing kernel Hilbert space. Section 4 is devoted to the study of a quaternionic unitary isomorphism between the slice Cholewinski-Fock space and a suitable quaternionic Hilbert space on the real line. This quaternionic isomorphism will be connected also to what we call the slice Dunkl transform. Then, section 5 deals with two right quaternionic linear operators that are proved to be adjoint of each other and satisfy a specific commutation rule on the slice Cholewinski-Fock space. Finally, the last section explains how the results obtained in this paper in the slice quaternionic setting could be extended in a similar way to the slice monogenic setting with Clifford algebras valued functions.

\section{Preliminaries}
 \subsection{Quaternions and slice regular functions}
For more details about the theory of slice regular functions and its different generalizations and applications one can see for example \cite{ColomboSabadiniStruppa2016,ColomboSabadiniStruppa2011,GentiliStruppa07,
GentiliStoppatoStruppa2013}. \\ \\
The non-commutative field of quaternions is defined to be
$$\Hq=\lbrace{q=x_0+x_1i+x_2j+x_3k\quad ; x_0,x_1,x_2,x_3\in\R}\rbrace$$ where the imaginary units satisfy the Hamiltonian multiplication rules $$i^2=j^2=k^2=-1\quad \text{and}\quad ij=-ji=k, jk=-kj=i, ki=-ik=j.$$
On $\Hq$ the conjugate and the modulus of $q$ are defined respectively by
$$\overline{q}=Re(q)-Im(q) \quad \text{where} \quad Re(q)=x_0, \quad Im(q)=x_1i+x_2j+x_3k$$
and $$\vert{q}\vert=\sqrt{q\overline{q}}=\sqrt{x_0^2+x_1^2+x_2^2+x_3^2}.$$
Note that the quaternionic conjugation satisfy the property $\overline{ pq }= \overline{q}\, \overline{p}$ for any $p,q\in \Hq$.
Moreover, the unit sphere $$S^2=\lbrace{q=x_1i+x_2j+x_3k;\text{ } x_1^2+x_2^2+x_3^2=1}\rbrace$$ can be identified with the set of all  imaginary units given by $$\mathbb{S}=\lbrace{q\in{\Hq};q^2=-1}\rbrace.$$
 It can be shown that any pure imaginary quaternion $q\in \Hq\setminus \R$ can be written in a unique way as $q=x+I y$ for some real numbers $x$ and $y>0$, and imaginary unit $I\in \mathbb{S}$.
Then, for every given $I\in{\mathbb{S}}$, the slice $\C_I$ is defined to be $\mathbb{R}+\mathbb{R}I$ and it is isomorphic to the complex plane $\C$ so that it can be considered as a complex plane in $\Hq$ passing through $0$, $1$ and $I$. Then, the union of all the slices of $\Hq$ is the whole space of quaternions $\Hq=\underset{I\in{\mathbb{S}}}{\cup}\C_I.$
\\ In \cite{GentiliStruppa07} the authors proposed a new definition to extend the classical theory of holomorphic functions in complex analysis to the quaternionic setting. This leads to the new theory of slice hyperholomorphic or slice regular functions on quaternions :
\begin{defn}
A real differentiable function $f: \Omega \longrightarrow \Hq$, on a given domain $\Omega\subset \Hq$, is said to be a slice (left) regular function if, for every $I\in \Sq$, the restriction $f_I$ to the slice $\C_{I}$, with variable $q=x+Iy$, is holomorphic on $\Omega_I := \Omega \cap \C_I$, that is it has continuous partial derivatives with respect to $x$ and $y$ and the function
$\overline{\partial_I} f : \Omega_I \longrightarrow \Hq$ defined by
$$
\overline{\partial_I} f(x+Iy):=
\dfrac{1}{2}\left(\frac{\partial }{\partial x}+I\frac{\partial }{\partial y}\right)f_I(x+yI)
$$
vanishes identically on $\Omega_I$. \\ The slice derivative $\partial_S f$ of $f$ is defined by :
\begin{equation*}
\partial_S(f)(q):=
\left\{
\begin{array}{rl}
\partial_I(f)(q)& \text{if } q=x+yI, y\neq 0\\
\displaystyle\frac{\partial}{\partial{x}}(f)(x) & \text{if } q=x \text{ is real}.
\end{array}
\right.
\end{equation*}
In addition to that we have :
\begin{enumerate}
\item A slice regular function on a domain $\Omega$ is said to be quaternionic intrinsic if $f(\Omega_I)\subset\C_I$ for any $I\in\mathbb{S}$.
\item  A function which is slice regular on the whole space of quaternions $\Hq$ is said to be entire slice regular.
\end{enumerate}
\end{defn}
We will refer to left slice regular functions as slice regular functions or simply regular functions for short and denote their space by $\mathcal{SR}(\Omega)$. It turns out that $\mathcal{S}\mathcal{R}(\Omega)$ is a right vector space over the noncommutative field $\Hq$.
\begin{rem}
The multiplication and composition of slice regular functions are not slice regular, in general. Moreover, the slice derivative does not satisfy the Leibniz rule with respect to the point wise multiplication. However, if $f$ and $g$ are slice regular functions such that $f$ is quaternionic intrinsic then the pointwise multiplication $h=fg$ is slice regular, see for example \cite{ColomboSabadiniStruppa2016}.
\end{rem}

 According to the last definition, the basic polynomials with quaternionic coefficients on the right are slice regular. Moreover, for any power series $\displaystyle\sum_n q^na_n$, there exists $0\leq R \leq \infty$, called the radius of convergence such that the power series defines a slice regular function on $B(0,R):= \{q\in \Hq; \, |q|<R\}$. The space of slice regular functions is endowed with the natural topology of uniform convergence on compact sets. Characterization of slice regular functions on a ball $B = B(0,R)$ centered at the origin is given by

\begin{thm}[Series expansion]
A given $\Hq$-valued function $f$ is slice regular on $B(0,R)\subset \Hq$ if and only if it has a series expansion of the form:
$$f(q)=\sum_{n=0}^{+\infty} q^n\frac{1}{n!}\partial^{(n)}_S(f)(0)$$
converging on $B(0,R)=\{q\in\Hq;\mid q\mid<R\}$.
\end{thm}

\begin{defn}
A domain $\Omega\subset \Hq$ is said to be a slice domain (or just $s$-domain) if  $\Omega\cap{\mathbb{R}}$ is nonempty and for all $I\in{\mathbb{S}}$, the set $\Omega_I:=\Omega\cap{\C_I}$ is a domain of the complex plane $\C_I$.
If moreover, for every $q=x+yI\in{\Omega}$, the whole sphere $x+y\mathbb{S}:=\lbrace{x+yJ; \, J\in{\mathbb{S}}}\rbrace$
is contained in $\Omega$, we say that  $\Omega$ is an axially symmetric slice domain.
\end{defn}

\begin{ex}
The whole space $\Hq$  and the Euclidean ball $B=B(0,R)$ of radius $R$ centered at the origin are axially symmetric slice domains.
\end{ex}
The following properties of slice regular functions are of particular interest and will be very useful for the next sections of this paper, see \cite{ColomboSabadiniStruppa2011,GentiliStoppatoStruppa2013}.

\begin{lem}[Splitting Lemma]\label{split} Let $f$ be a slice regular function on a domain $\Omega$. Then, for every $I$ and $J$ two perpendicular imaginary units there exist two holomorphic functions $F,G:\Omega_{I}\longrightarrow{\C_I}$ such that for all $z=x+yI\in{\Omega_I}$, we have
$$f_I(z)=F(z)+G(z)J,$$
 where $\Omega_I=\Omega\cap{\C_I}$ and $\C_I=\mathbb{R}+\mathbb{R}I.$
\end{lem}

\begin{thm}[Representation Formula]\label{repform}
Let $\Omega$ be an axially symmetric slice domain and $f\in{\mathcal{SR}(\Omega)}$. Then, for any $I,J\in{\mathbb{S}}$, we have the formula
$$
f(x+yJ)= \frac{1}{2}(1-JI)f_I(x+yI)+ \frac{1}{2}(1+JI)f_I(x-yI)
$$
for all $q=x+yJ\in{\Omega}$.
\end{thm}

\begin{lem}[Extension Lemma]\label{extensionLem}
Let $\Omega_I$ be a domain in $\C_I$ symmetric with respect to the real axis and such that $\Omega_I\cap\R\neq \emptyset$. Let $h:\Omega_I\longrightarrow \Hq$ be an holomorphic function. Then, the function $ext(h)$ defined by
$$ext(h)(x+yJ):= \dfrac{1}{2}[h(x+yI)+h(x-yI)]+\frac{JI}{2}[h(x-yI)-h(x+yI)];  \quad J\in \mathbb{S},$$
extends $h$ to a regular function on $\overset{\sim}\Omega=\underset{x+yJ\in{\Omega}}\cup x+y\mathbb{S}$, the symmetric completion of $\Omega_I$.
Moreover, $ext(h)$ is the unique slice regular extension of $h$.
\end{lem}
\begin{thm}[Identity Principle] \label{IdentityPrin}
Let $f$ and $g$ be two slice regular functions on a slice
domain $\Omega$. If, for some $I\in\mathbb{S}$, $f$ and $g$ coincide on a subset of $\Omega_I$ having an accumulation point in $\Omega_I$, then $f=g$ on the whole domain $\Omega$.
\end{thm}

\subsection{Some properties of Bessel and modified Bessel functions}
For more details about the subject of Bessel functions and related topics we refer the reader to \cite{Erdelyi1953,Lebedev1972}. \\ \\
To any complex number $\nu$ is associated the so-called Bessel's differential equation \begin{equation}
x^2\frac{d^2}{dx^2}y+x\frac{d}{dx}y+(x^2-\nu^2)y=0
\end{equation}
Using the Frobeinus method, a solution of the last equation is given by the Bessel's function of the first kind, namely
$$J_\nu(x):=\displaystyle\left(\frac{x}{2}\right)^\nu\sum_{k=0}^\infty \frac{(-1)^k}{\Gamma(k+1)\Gamma(\nu+k+1)}\left(\frac{x}{2}\right)^{2k}.$$

The second linear independent solution of the Bessel's equation is the Bessel function of the second kind $Y_\nu$ which is defined by $$Y_\nu(x)=\frac{\cos(\nu\pi)J_{\nu}(x)-J_{-\nu}(x)}{\sin(\nu\pi)}\quad \text{if} \quad \nu \notin \Z \quad $$
and $$Y_n(x)=\underset{\nu\rightarrow n}\lim Y_\nu(x) \quad \text{if} \quad \nu=n\in\Z.$$
The same reasoning is adopted to construct a modified Bessel function of the third kind sometimes called also the Macdonald's function and denoted by $K_\nu(x)$. To this end, we consider the modified Bessel's equation given by \begin{equation}
x^2\frac{d^2}{dx^2}y+x\frac{d}{dx}y-(x^2+\nu^2)y=0
\end{equation}
Analogously, the modified Bessel function of the first kind is defined by $$I_\nu(x):=\displaystyle\left(\frac{x}{2}\right)^\nu\sum_{k=0}^\infty \frac{1}{\Gamma(k+1)\Gamma(\nu+k+1)}\left(\frac{x}{2}\right)^{2k}$$ and the Macdonald's function is defined by $$\displaystyle K_\nu(x)=\frac{\pi}{2}\frac{I_{-\nu}(x)-I_{\nu}(x)}{\sin(\nu\pi)}\quad\text{if} \quad \nu \notin \Z $$
and $$K_n(x)=\underset{\nu\rightarrow n}\lim K_\nu(x) \quad \text{if} \quad\nu=n\in\Z.$$
\newline The Macdonald's function is of a particular interest for our study since it will appear in the next section as a weight of the quaternionic Hilbert space of entire slice regular functions instead of the classical Gaussian measure. \\ \\ So, we summarize in the following Proposition some interesting properties of this function that will be useful in the sequel, see \cite{Erdelyi1953,Lebedev1972}. \\
\begin{prop} \label{Mac}
Let $x>0$ and $\delta,\nu\in \R$ such that $\delta+\nu >0$ and $\delta-\nu >0$.
Then, we have the following formulas
\begin{enumerate}
\item $\displaystyle K_\nu(x)=\int_0^\infty \exp(-x\cosh t)\cosh(\nu t)dt.$
\item $\displaystyle K_{\frac{1}{2}}(x)=K_{-\frac{1}{2}}(x)=\sqrt{\frac{\pi}{2x}}e^{-x}.$
\item $\displaystyle\int_0^\infty t^{\delta-1}K_\nu(t)dt=2^{\delta-2}\Gamma\left(\frac{\delta}{2}+\frac{\nu}{2}\right)\Gamma\left(\frac{\delta}{2}-\frac{\nu}{2}\right).$

\end{enumerate}
\end{prop}
\section{The slice hyperholomorphic Cholewinski-Fock space}
Any quaternionic entire function may be written as $$f=f^e+f^o$$ $f^e$ and $f^o$ are respectively even and odd functions where $$f^e(q):=\displaystyle\frac{f(q)+f(-q)}{2} \quad \text{and} \quad f^o(q):=\frac{f(q)-f(-q)}{2}.$$
Then, thanks to the series expansion theorem for slice regular functions we have $$f(q)=\displaystyle\sum^\infty_{n=0} q^na_n \quad \text{with} \quad a_n\in\Hq$$ so that, $$f^e(q)=\displaystyle\sum_{n=0}^\infty q^{2n}a_{2n} \quad \text{and} \quad f^o(q)=\displaystyle\sum_{n=0}^\infty q^{2n+1}a_{2n+1}.$$
Now, let $\displaystyle \alpha\geq-\frac{1}{2}$ and $I$ be any imaginary unit in the sphere $\mathbb{S}$. Then, for $p=x+yI$ in the slice $\C_I$ we consider the following probability measure $$\displaystyle d\lambda_{\alpha,I}(p):=\frac{\vert{p}\vert^{2\alpha+2}}{\pi 2^{\alpha}\Gamma(\alpha+1)}K_\alpha(\vert{p}\vert^2)d\lambda_I(p)$$
where $K_\alpha$ is the Macdonald function and $d\lambda_I(p)$ is the usual Lebesgue measure on the slice $\C_I$.  In \cite{SifiSoltani2002} the complex generalized Fock space $\mathcal{F}^\alpha(\C)$ was defined to be the space consisting of complex entire functions $f:\C\longrightarrow\C$ satisfying: $$\displaystyle \int_{\C}\vert{f^e(z)}\vert^2 d \lambda_{\alpha}(z)+2(\alpha+1)\int_{\C}\vert{f^o(z)}\vert^2\vert{z}\vert^{-2}d\lambda_{\alpha+1}(z)<\infty.$$
Then,  we consider the following definition
\begin{defn}
A slice entire function $f:\Hq\longrightarrow \Hq$ is said to be in the slice Cholewinski-Fock space or the generalized slice Fock space if, for $I\in\Sq$ it satisfies the following condition
$$\displaystyle \int_{\C_I}\vert{f^e_I(p)}\vert^2 d \lambda_{\alpha,I}(p)+2(\alpha+1)\int_{\C_I}\vert{f^o_I(p)}\vert^2\vert{p}\vert^{-2}d\lambda_{\alpha+1,I}(p)<\infty.$$
The space containing all such functions will be denoted $\mathcal{F}_{Slice}^\alpha(\Hq)$.
\end{defn}
\begin{rem}
Notice that if $\displaystyle\alpha=-\frac{1}{2}$ then thanks to (2) in Proposition \ref{Mac} we can see that $\CFH$ is exactly the slice hyperholomorphic Fock space introduced and studied in \cite{AlpayColomboSabadini2014}. Indeed, for $\displaystyle\alpha=-\frac{1}{2}$ we get $$\displaystyle d\lambda_{\alpha,I}(p):=\frac{1}{2\pi}e^{-\vert p \vert^{2}}d\lambda_I(p).$$
In particular, in this case $f$ belongs to $\CFH$  if and only if it belongs to the classical slice hyperholomorphic Fock space.
\end{rem}

For $f,g\in \CFH$ we define the following inner product $$\displaystyle\scal{f,g}_{\CFH}:=\int_{\C_I}\overline{g^e_I(p)}f^e_I(p)d\lambda_{\alpha,I}(p)+2(\alpha+1)\int_{\C_I}\overline{g^o_I(p)}f^o_I(p)\vert{p}\vert^{-2}d\lambda_{\alpha+1,I}(p)$$

We shall see later that this definition is well posed since it does not depend on the choice of the imaginary unit $I$. We have :
\begin{prop}
$\CFH$ is a right quaternionic Hilbert space with respect to $\scal{.,.}_{\CFH}$.
\end{prop}
\begin{proof}
Let $(f_n)$ be a Cauchy sequence in $\CFH$. Take $I,J\in\mathbb{S}$ such that $I\perp J$. Then, since $f_n$ are slice regular we can use the Splitting Lemma to write $$f_{n,I}:=F_n+G_nJ \quad \forall n\in\N$$ where $F_n$ and $G_n$ are holomorphic functions on the slice $\C_I$ belonging to the generalized complex Fock space $\mathcal{F}^\alpha(\C_I)$. It is easy to see that $(F_n)_n$ and $(G_n)_n$ are Cauchy sequences in $\mathcal{F}^\alpha(\C_I)$. Hence, there exists two functions $F$ and $G$ belonging to $\mathcal{F}^\alpha(\C_I)$ such that the sequences $(F_n)_n$ and $(G_n)_n$ are converging respectively to $F$ and $G$. Let $f_I=F+GJ$ and consider $f=ext(f_I)$ we have then $f\in \CFH$. Moreover, the sequence $(f_n)$ converges to $f$ with respect to the norm of $\CFH$. This ends the proof.
\end{proof}

For any $m,n\geq 0,$ we set $$\displaystyle E_{m,n}(\alpha):=\int_{\C_I}\overline{q^{2m}}q^{2n}d\lambda_{\alpha,I}(q)$$
and $$O_{m,n}(\alpha):=\int_{\C_I}\overline{q^{2m+1}}q^{2n+1}\vert{q}\vert^{-2}   \lambda_{\alpha+1,I}(q).$$

Then, the following formulas hold
\begin{lem} \label{comp} For all $m,n\geq 0,$ we have
\begin{enumerate}
\item[i)] $\displaystyle E_{m,n}(\alpha)=\delta_{m,n}2^{2n}n!\frac{\Gamma(\alpha+n+1)}{\Gamma(\alpha+1)}$.
\item [ii)] $\displaystyle O_{m,n}(\alpha)=E_{m,n}(\alpha+1)$.
\end{enumerate}
\end{lem}
\begin{proof}
\begin{enumerate}
\item[i)] We write $q=re^{I\theta}$ using the polar coordinates.
This leads to \begin{align*}
E_{m,n}(\alpha)
 &=\displaystyle \frac{1}{2^\alpha\pi\Gamma(\alpha+1)}\int_0^\infty \int_0^{2\pi}e^{2(n-m)\theta I}r^{2(m+n+\alpha+1)}K_\alpha(r^2)rdrd\theta
 \\& =\frac{1}{2^\alpha\pi\Gamma(\alpha+1)}\int_0^{2\pi}e^{2(n-m)\theta I}d\theta \int_0^\infty r^{2(m+n+\alpha+1)}K_\alpha(r^2)rdr
 \\& = \frac{2\delta_{m,n}}{2^\alpha\Gamma(\alpha+1)}\int_0^\infty r^{2(2n+\alpha+1)}K_{\alpha}(r^2)rdr.
 \end{align*}
 Making use of the change of variable $t=r^2,$ we obtain $$E_{m,n}(\alpha)=\displaystyle\frac{\delta_{m,n}}{2^\alpha\Gamma(\alpha+1)}\int_0^\infty t^{2n+\alpha+1}K_\alpha(t)dt$$
 Then, the proof of i) ends thanks to the property 3 in Proposition \ref{Mac} by taking $\delta=2n+\alpha+2.$
 \item[ii)] This is obvious from the definition of $O_{m,n}(\alpha)$.
\end{enumerate}
\end{proof}

Thanks to the last lemma we have the two following propositions :
\begin{prop}\label{scal}
Let $f(q)=\displaystyle\sum_{n=0}^\infty q^na_n$ and $g(q)=\displaystyle\sum_{n=0}^\infty q^nb_n$ be two slice regular functions belonging to $\CFH$. Then, we have $$\scal{f,g}_{\CFH}=\displaystyle \sum_{n=0}^\infty \overline{b_n}a_n\beta_n(\alpha)$$
where $$\displaystyle\beta_n(\alpha):=2^n\left[\frac{n}{2}\right]!\frac{\displaystyle\Gamma\left(\left[\frac{n+1}{2}\right]+\alpha+1\right)}{\Gamma(\alpha+1)}.$$
Here the symbol $\left[.\right]$ stands for the integer part.
\end{prop}
\begin{proof}
We have $$\displaystyle\scal{f,g}_{\CFH}:=\int_{\C_I}\overline{g^e_I(p)}f^e_I(p)d\lambda_{\alpha,I}(p)+2(\alpha+1)\int_{\C_I}\overline{g_o^I(p)}f^o_I(p)\vert{p}\vert^{-2}d\lambda_{\alpha+1,I}(p)$$
Then, we set $$A:= \int_{\C_I}\overline{g^e_I(p)}f^e_I(p)d\lambda_{\alpha,I}(p)\quad \text{and} \quad B:=2(\alpha+1)\int_{\C_I}\overline{g^o_I(p)}f^o_I(p)\vert{p}\vert^{-2}d\lambda_{\alpha+1,I}(p)$$
Notice that
$$f^e(q)=\displaystyle\sum_{n=0}^\infty q^{2n}a_{2n} \quad \text{and} \quad g^e(q)=\sum_{m=0}^\infty q^{2m}b_{2m}.$$
Thus, \begin{align*}
A
 &=\displaystyle \lim_{R\rightarrow\infty} \int_{\lbrace{\vert{p}\vert<R}\rbrace}\left(\sum_{m=0}^\infty \overline{b_{2m}}\overline{p^{2m}}\right)\left(\sum_{n=0}^\infty p^{2n}a_{2n}\right)d\lambda_{\alpha,I}(p)
 \\& = \sum_{m,n=0}^\infty\overline{b_{2m}}E_{m,n}(\alpha)a_{2n}.
 \\&
 \end{align*}
Hence, making use of the Lemma \ref{comp} we get $$A=\displaystyle\sum_{k=0}^\infty\overline{b_{2k}}a_{2k}\beta_{2k}(\alpha).$$

 Similarly, by writing $$f^o(q)=\displaystyle\sum_{n=0}^\infty q^{2n+1}a_{2n+1} \quad \text{and} \quad g^o(q)=\sum_{m=0}^\infty q^{2m+1}b_{2m+1}$$
 we obtain $$ B=\displaystyle\sum_{k=0}^\infty\overline{b_{2k+1}}a_{2k+1}\beta_{2k+1}(\alpha).$$

 This leads to $$\scal{f,g}_{\CFH}=\displaystyle \sum_{n=0}^\infty \overline{b_n}a_n\beta_n(\alpha).$$
\end{proof}
\begin{rem}
Proposition \ref{scal} shows that the scalar product is independent of the choice of the imaginary unit $I$.
\end{rem}
\begin{prop}
For any $n\in\N$, we consider the functions $$\phi_n^\alpha(q)=\displaystyle\frac{q^n}{\sqrt{\beta_n(\alpha)}}.$$ Then, $\lbrace{\phi_n^\alpha}\rbrace_n$ form an orthonormal basis of $\CFH$.
\end{prop}
\begin{proof}
Lemma \ref{comp} allows us to easily check that $$\scal{\phi_m^\alpha,\phi_n^\alpha}_{\CFH}=\delta_{m,n} \quad \forall m,n\in \N.$$
Let us prove now that these functions form a basis of $\CFH$. Indeed, take $f$ in $\CFH$ such that $\scal{f,\phi_k^\alpha}_{\CFH}=0,\forall k$. Then, since $f$ is entire slice regular it admits a series expansion so that $f=\displaystyle\sum_{n=0}^\infty\phi_n^\alpha c_n$ where $(c_n)_n\subset\Hq.$ Notice that from the Proposition \ref{scal} we have $$\scal{f,\phi_k^\alpha}_{\CFH}=c_k, \forall k\geq0.$$
This shows that $f$ is identically zero.
\end{proof}
An immediate consequence is :
\begin{cor} An entire function of the form $\displaystyle f(q)=\sum_{n=0}^\infty q^na_n$ belongs to $\CFH$ if and only if it satisfies the following growth condition $$\displaystyle \sum_{n=0}^\infty \vert{a_n}\vert^2\beta_n(\alpha)<\infty. $$
\end{cor}

\begin{lem} \label{beta}
For all $n\in\N,$ we have $n!\leq \beta_n(\alpha).$
\end{lem}
\begin{proof}
This is a consequence of the Duplication formula for the Gamma function given by 
$$\frac{\Gamma(x)\Gamma(x+\frac{1}{2})}{\Gamma(2x)}=\frac{\sqrt{\pi}}{2^{2x-1}},$$combined with the fact that the function $\beta_n$ is increasing for $\alpha\geq \displaystyle -\frac{1}{2}.$ Indeed, by treating both cases of $n=2k$ and $n=2k+1$ with $k\in \mathbb{N}$ using the Duplication formula we get $$\displaystyle \beta_n\left(-\frac{1}{2}\right)=n!\leq \beta_n(\alpha).$$
\end{proof}
\begin{rem}
$\CFH$ is continuously embedded in the slice Fock space $\mathcal{F}_{Slice}(\Hq)$. Indeed, we have $$\Vert{f}\Vert_{\mathcal{F}_{Slice}(\Hq)}^2=\displaystyle \sum_{n=0}^\infty \vert{a_n}\vert^2n!\leq \displaystyle \sum_{n=0}^\infty \vert{a_n}\vert^2\beta_n(\alpha)= \Vert{f}\Vert_{\CFH}^2.$$ The equality holds if $\displaystyle\alpha=-\frac{1}{2}.$
\end{rem}
In the sequel, we shall prove that $\CFH$ is a quaternionic reproducing kernel Hilbert space and give an expression of its reproducing kernel. To this end, let us fix $q\in\Hq$ and consider the evaluation mapping $$\Lambda_q:\CFH\longrightarrow \Hq; f\mapsto\Lambda_q(f)=f(q).$$ Then, we have the following estimate on $\CFH$ :

\begin{prop}\label{estimate2}
For any $f\in{\CFH},$ there exists $0<C(\vert{q}\vert)\leq e^{\frac{\vert{q}\vert^2}{2}}$ such that $$\vert{\Lambda_q(f)}\vert=\vert{f(q)}\vert\leq C(\vert{q}\vert) \norm{f}_{\CFH}.$$
\end{prop}
\begin{proof}
The series expansion theorem for slice regular functions asserts that $$f(q)=\displaystyle\sum_{n=0}^\infty q^na_n \quad \text{with} \quad (a_n)_n\subset \Hq.$$
Then using the Cauchy-Schwarz inequality we have the following estimates, \begin{align*}
\vert{\Lambda_q(f)}\vert=\vert{f(q)}\vert & \leq{ \sum_{n=0}^{\infty}\vert{q}\vert^n\vert{a_n}\vert} \\&
                    \leq \sum_{n=0}^{\infty} \frac{\vert{q}\vert^n}{\sqrt{\beta_n(\alpha)}} \sqrt{\beta_n(\alpha)}\vert{a_n}\vert
\\& \leq  \left( \sum_{n=0}^{\infty}\frac{\vert{q}\vert^{2n}}{\beta_n(\alpha)} \right)^{\frac{1}{2}}\left( \sum_{n=0}^{\infty}\beta_n(\alpha)\vert{a_n}\vert^2\right)^\frac{1}{2}
\\&  \leq C(\vert{q}\vert)\Vert{f}\Vert_{\CFH}.
\end{align*}

Notice that thanks to Lemma \ref{beta} this constant could be also estimated so that we have $$C(\vert{q}\vert)=\left( \sum_{n=0}^{\infty}\frac{\vert{q}\vert^{2n}}{\beta_n(\alpha)} \right)^{\frac{1}{2}}\leq e^{\frac{\vert{q}\vert^2}{2}}.$$
\end{proof}

\begin{rem}
Proposition \ref{estimate2} shows that all the evaluation mappings on $\CFH$ are continuous. Then, the Riesz representation theorem for quaternionic Hilbert spaces, see \cite{BDS1982} asserts that $\CFH$ is a quaternionic reproducing kernel Hilbert space.
\end{rem}
 For $p,q\in\Hq,$ we consider the function $$\displaystyle L_\alpha(p,q)=L_\alpha^q(p):=\sum_{n=0}^\infty\frac{p^nq^n}{\beta_n(\alpha)}$$
where $$\displaystyle\beta_n(\alpha):=2^n\left[\frac{n}{2}\right]!\frac{\Gamma\left(\left[\frac{n+1}{2}\right]+\alpha+1\right)}{\Gamma(\alpha+1)}.$$
If $q=x\in\R$ we use the following notation $L_\alpha(p,x)=L_\alpha(px)$ since $px=xp$. \\ \\ The function $$K_\alpha(p,q)=K_\alpha^q(p):=L_\alpha^{\overline{q}}(p)$$ satisfies the following properties
\begin{prop} \label{rkp}
\begin{enumerate} Let $q,s\in\Hq$ fixed,
\item[(a)] $K_q^\alpha\in\CFH$.
\item[(b)] For all $f\in \CFH,$ we have $f(q)=\scal{f,K_q^\alpha}_{\CFH}.$
\item[(c)] $\scal{K_q^\alpha,K^\alpha_s}_{\CFH}=K^\alpha_q(s).$
\end{enumerate}
\end{prop}
\begin{proof}
\begin{enumerate}
\item We have by definition $$\displaystyle K_\alpha^q(p):=\sum_{n=0}^\infty p^na_n \quad \text{where} \quad a_n=\frac{\overline{q}^n}{\beta_n(\alpha)}.$$
Thus,
\begin{align*}
\displaystyle\sum_{n=0}^\infty\beta_n(\alpha)\vert{a_n}\vert^2 &=\sum_{n=0}^\infty \frac{\vert{q}\vert^{2n}}{\beta_n(\alpha)}  \\&
                    = C^2(\vert{q}\vert)
\\& \leq  e^{\vert{q}\vert^2}< \infty.
\\&
\end{align*}
\item Let $f\in \CFH$; since $f$ is slice regular on $\Hq$ we can write $f(p)=\displaystyle \sum_{n=0}^\infty p^nb_n \quad \text{with}\quad (b_n)_n\subset\Hq.$ Then, using the expression of $K_q^\alpha$ combined with Proposition \ref{scal} we obtain $$\scal{f,K_q^\alpha}_{\CFH}=\sum_{n=0}^\infty q^nb_n=f(q).$$
\item We just need to write $$K_q^\alpha(p)=\displaystyle\sum_{n=0}^\infty\frac{\overline{q}^n}{\beta_n(\alpha)}p^n \quad \text{and} \quad K_s^\alpha(p)=\displaystyle\sum_{n=0}^\infty\frac{\overline{s}^n}{\beta_n(\alpha)}p^n$$ and then use the Proposition \ref{scal}.
\end{enumerate}
\end{proof}
The results obtained in this section may be summarized in the following
\begin{thm}
$\CFH$ is a quaternionic reproducing kernel Hilbert space whose reproducing kernel is given by the following formula $$K_\alpha(p,q)=\sum_{n=0}^\infty\frac{p^n\overline{q^n}}{\beta_n(\alpha)} \quad \text{for all } (p,q)\in\Hq\times\Hq.$$
\end{thm}
\begin{proof}
It is a consequence of the Propositions \ref{estimate2} and \ref{rkp}.
\end{proof}
\begin{rem}
For $\displaystyle\alpha=-\frac{1}{2}$, it turns out that $K_\alpha(.,.)$ is the reproducing kernel of the slice Fock space $\mathcal{F}_{Slice}(\Hq)$ obtained in \cite{AlpayColomboSabadini2014} and given by $$e_*(p\overline{q})=\displaystyle\sum_{n=0}^\infty\frac{p^n\overline{q}^n}{n!}.$$
\end{rem}
\section{A unitary integral transform associated to $\CFH$}
In this section we introduce an integral operator $T_\alpha$ and show that it defines an isometric isomorphism between the quaternionic slice Cholewinski Fock space introduced in the last section and a specific quaternionic Hilbert space on the real line, namely $\mathcal{H}_\alpha=L^2_\Hq(\R,d\mu_\alpha)$ where

$$\displaystyle d\mu_\alpha(x):=\frac{\vert{x}\vert^{2\alpha+1}}{2^{\alpha+1}\Gamma(\alpha+1)}dx.$$
Note that the quaternionic Hilbert space $\mathcal{H}_\alpha$ is endowed by the inner product $$\displaystyle\scal{\psi,\phi}_{\mathcal{H}_\alpha}:=\int_\R \overline{\phi(x)}\psi(x)d\mu_\alpha(x).$$
 Note that $\mathcal{H}_{-\frac{1}{2}}$ is the standard Hilbert space $L^2_\Hq(\R)$. In this case, the isomorphism $T_\alpha$ is the quaternionic Segal-Bargmann transform introduced and studied in \cite{DG2017}. \\ \\
 Let us consider the kernel
$$\displaystyle\mathcal{C}_\alpha(p,x):=2^{\frac{\alpha+1}{2}}e^{-\frac{1}{2}(p^2+x^2)}L_\alpha(\sqrt{2}px) \quad \forall (p,x)\in \Hq\times\R$$
so that, for $\varphi\in\mathcal{H}_\alpha$ and $q\in\Hq$ we define $$\displaystyle T_\alpha\varphi(q):=\int_\R \mathcal{C}_\alpha(q,x)\varphi(x)d\mu_\alpha(x).$$
In the sequel, we study the integral transform $T_\alpha$. To this end, let us recall some properties of the so called generalized Hermite polynomials, see \cite{Rosenblum1994,Rosler1998,SifiSoltani2002}. \\ These generalized Hermite polynomials are defined by

$$H_n^\alpha(x):=\displaystyle \sum_{k=0}^{[\frac{n}{2}]}\frac{(-1)^k\beta_n(\alpha)}{k!\beta_{n-k}(\alpha)}(2x)^{n-2k}; x\in\R$$ their generating function in the classical complex case is given by

$$\displaystyle e^{-\frac{z^2}{4}}L_\alpha(zx)=\sum_{n=0}^\infty\frac{H_{n}^{\alpha}(x)}{2^n\beta_n(\alpha)}z^n; \quad z\in\C,x\in\R.$$
Moreover, it turns out that the Hermite functions $$h_n^\alpha(x):=\displaystyle \frac{2^{-\frac{(n-\alpha-1)}{2}}}{\sqrt{\beta_n(\alpha)}}e^{-\frac{x^2}{2}}H_n^\alpha(x)$$ associated to these polynomials form an orthonormal basis of the Hilbert space $\mathcal{H}_\alpha$. \\ \\
Then, similarly to the complex case we prove the following
\begin{lem}
Let $p\in\Hq$ and $x\in\R$. Then, we have $$\displaystyle e^{-\frac{p^2}{4}}L_\alpha(px)=\sum_{n=0}^\infty\frac{H_{n}^{\alpha}(x)}{2^n\beta_n(\alpha)}p^n.$$
\end{lem}
\begin{proof}
Set $$f_i(z)=e^{-\frac{z^2}{4}}L_\alpha(zx) \quad \text{and} \quad  g_i(z)=\displaystyle\sum_{n=0}^\infty \frac{H_n^\alpha(x)}{2^n\beta_n(\alpha)}z^n;\quad \forall z\in\C_i. $$  Notice that $f_i$ and $g_i$ are two entire functions. Then, they could be extended into two slice entire regular functions denoted respectively $ext(f_i)$ and $ext(g_i)$. On the other hand, the functions $$F(q)=e^{-\frac{q^2}{4}}L_\alpha(qx) \quad \text{and} \quad G(q)=\displaystyle\sum_{n=0}^\infty \frac{H_n^\alpha(x)}{2^n\beta_n(\alpha)}q^n$$ are entire slice regular on $\Hq$ since $q\mapsto e^{-\frac{q^2}{4}}$ is quaternionic intrinsic. It follows then from the uniqueness in the Lemma \ref{extensionLem} that $F=ext(f_i)$ and $G=ext(g_i)$. Finally, since $F$ and $G$ coincide on the slice $\C_i$ we use the identity principle to conclude the proof.
\end{proof}
As a consequence of the last lemma we have
\begin{prop}\label{GF}
Let $p\in\Hq$ and $x\in\R$. Then, $$\displaystyle\mathcal{C}_\alpha(p,x)=\sum_{n=0}^\infty h_n^\alpha(x) \frac{p^n}{\sqrt{\beta_n(\alpha)}}.$$
\end{prop}
\begin{proof}
Let $(p,x)\in\Hq\times\R,$ and write $$\displaystyle\sum_{n=0}^\infty h_n^\alpha(x) \frac{p^n}{\sqrt{\beta_n(\alpha)}}=2^{\frac{\alpha+1}{2}}e^{-\frac{x^2}{2}}\sum_{n=0}^\infty \frac{H_n^\alpha(x)}{\beta_n(\alpha)}\left(\frac{\sqrt{2}p}{2}\right)^n$$
Then, by taking $q=\sqrt{2}p\in\Hq$ we can apply the last lemma to get
\begin{align*}
\displaystyle\sum_{n=0}^\infty h_n^\alpha(x) \frac{p^n}{\sqrt{\beta_n(\alpha)}}
 &=2^{\frac{\alpha+1}{2}}e^{-\frac{1}{2}(x^2+p^2)}L_\alpha(\sqrt{2}px)
 \\& =\mathcal{C}_\alpha(p,x).
 \\
 \end{align*}
\end{proof}
Another interesting property of the kernel $\mathcal{C}_\alpha$ is given by
\begin{lem} \label{ker}
Let $p,q\in\Hq.$ Then, $$\displaystyle\int_\R \mathcal{C}_\alpha(p,x)\mathcal{C}_\alpha(q,x)d\mu_\alpha(x)=L_\alpha(p,q).$$
\end{lem}
\begin{proof}
Let $p,q\in\Hq,$ making use of the Proposition \ref{GF} we can write
$$\displaystyle\mathcal{C}_\alpha(p,x)=\sum_{n=0}^\infty h_n^\alpha(x) \frac{p^n}{\sqrt{\beta_n(\alpha)}}\quad \text{and} \quad \displaystyle\mathcal{C}_\alpha(q,x)=\sum_{n=0}^\infty h_n^\alpha(x) \frac{q^n}{\sqrt{\beta_n(\alpha)}}.$$
Thus, \begin{align*}
\displaystyle \int_\R\mathcal{C}_\alpha(p,x)\mathcal{C}_\alpha(q,x)d\mu_\alpha(x)
 &=\int_\R \left(\sum_{n=0}^\infty h_n^\alpha(x) \frac{p^n}{\sqrt{\beta_n(\alpha)}}\right)\left(\sum_{m=0}^\infty h_m^\alpha(x) \frac{q^m}{\sqrt{\beta_m(\alpha)}}\right)d\mu_\alpha(x)
 \\& =\sum_{n,m=0}^\infty\frac{p^nq^m}{\sqrt{\beta_n(\alpha)}\sqrt{\beta_m(\alpha)}}\int_\R h_n^\alpha(x)h_m^\alpha(x)d\mu_\alpha(x).
 \\
 \end{align*}
Then, since $\lbrace{h_n^\alpha}\rbrace_{n\geq0}$ form an orthonormal set in $\mathcal{H}_\alpha$ we have \begin{align*}
\displaystyle \int_\R\mathcal{C}_\alpha(p,x)\mathcal{C}_\alpha(q,x)d\mu_\alpha(x)
 &=\sum_{n,m=0}^\infty\frac{p^nq^m}{\sqrt{\beta_n(\alpha)}\sqrt{\beta_m(\alpha)}}\delta_{n,m}
 \\& =L_{\alpha}(p,q).
 \\
 \end{align*}
\end{proof}
Therefore, we have
\begin{prop} \label{estimation} Let $q\in\Hq$ and $\varphi\in\mathcal{H}_\alpha$. Then,
\begin{enumerate}
\item $\Vert{\mathcal{C}_\alpha^q}\Vert_{\mathcal{H}_\alpha}=\sqrt{L_\alpha(q,\overline{q})}=\sqrt{L_\alpha(\vert{q}\vert^2)}.$
\item $\vert{T_\alpha\varphi(q)}\vert\leq \sqrt{L_\alpha(\vert{q}\vert^2)}\Vert{\varphi}\Vert_{\mathcal{H}_\alpha}. $
\end{enumerate}
\end{prop}
\begin{proof}
\begin{enumerate}
\item It is a direct consequence of the Lemma \ref{ker} combined with the identity $\overline{\mathcal{C}_\alpha(q,x)}=\mathcal{C}_\alpha(\overline{q},x)$ for all $q\in\Hq.$
\item We start by writing $$\displaystyle \vert{T_\alpha\varphi(q)}\vert \leq\int_\R\vert{\mathcal{C}_\alpha(q,x)}\vert\vert{\varphi(x)}\vert d\mu_\alpha(x)$$
Then, we use the Cauchy-Schwarz inequality to complete the proof.
\end{enumerate}
\end{proof}
\begin{rem}
For $n\in\N,$ we have $T_\alpha h_n^\alpha=\phi_n^\alpha$ and $\Vert{T_\alpha h_n^\alpha }\Vert_{\CFH}=\Vert{h_n^\alpha}\Vert_{\mathcal{H}_\alpha}.$ Moreover, if $\displaystyle\varphi=\sum_{n=0}^N h_n^\alpha c_n$ then $\displaystyle\Vert{T_{\alpha}\varphi}\Vert_{\CFH}=\sum_{n=0}^N\vert{c_n}\vert^2=\Vert{\varphi}\Vert_{\mathcal{H}_{\alpha}}$.
\end{rem}
Finally, we prove the following
\begin{thm}
The integral transform $T_\alpha$ is an isometric isomorphism mapping the quaternionic Hilbert space $\mathcal{H}_\alpha$ onto $\CFH$.
\end{thm}
\begin{proof}
Let $\varphi\in\mathcal{H}_\alpha$ and set $\displaystyle\varphi_N=\sum_{n=0}^N h_n^\alpha a_n$ with $(a_n)\subset\Hq$ such that $\varphi_N$ converges to $\varphi$ in $\mathcal{H}_\alpha$. Then, making use of the second estimate in Proposition \ref{estimation} we can show that $T_\alpha\varphi_N$ is a Cauchy sequence in the quaternionic Hilbert space $\CFH$. Thus, there exists $f\in\CFH$ such that $T_\alpha\varphi_N \underset{\CFH}\longrightarrow f$. Consequently, it will exist a subsequence $(T_\alpha\varphi_{N_k})_{N_k}$ converging to $f$ pointwise almost everywhere. Moreover, according to the Proposition \ref{estimation} we have the following $$\vert{T_\alpha\varphi(p)-T_\alpha\varphi_N(p)}\vert\leq C(\vert{p}\vert)\Vert{\varphi-\varphi_N}\Vert_{\mathcal{H}_\alpha}$$
Then, by letting $N$ goes to infinity we can see that $(T_\alpha\varphi_N)_{N}$ converges pointwise to $T_\alpha\varphi$. In particular, the pointwise convergence shows that  $T_\alpha\varphi=f$. However, by definition we have $ f:=\underset{N\rightarrow\infty, \CFH}\lim T_\alpha\varphi_N.$

Therefore, it follows that $$\Vert{T_\alpha\varphi}\Vert_{\CFH}=\Vert{f}\Vert_{\CFH}=\underset{N\rightarrow\infty}\lim\Vert{T_\alpha\varphi_N}\Vert_{\CFH}=\Vert{\varphi}\Vert_{\mathcal{H}_\alpha}.$$ Hence, $T_\alpha$ is a quaternionic isometric integral operator which is one-to-one. Moreover, since $T_{\alpha}h_{n}^{\alpha}=\phi_{n}^{\alpha}$ it is also surjective and this ends the proof.
\end{proof}
As a consequence we have
\begin{prop}\label{inv}
Let $f\in\CFH$ and $x\in\R$. Then, for any imaginary unit $I\in\mathbb{S}$ we have $$\displaystyle T^{-1}_\alpha f(x)=\int_{\C_I}\mathcal{C}_{\alpha,I}^e(\overline{p},x)f^I_e(p)d\lambda_{\alpha,I}(p)+2(\alpha+1)\int_{\C_I}\mathcal{C}^o_{\alpha,I}(\overline{p},x)f^I_o(p)\vert{p}\vert^{-2}d\lambda_{\alpha+1,I}(p).$$
\end{prop}

\begin{proof}
First, note that $T_\alpha$ is a surjective isometry. Then, it defines a quaternionic unitary operator which means that its inverse is given by $T_\alpha^{-1}=T_\alpha^*$ where $T_\alpha^*$ is the adjoint operator of $T_\alpha$. 
Now, let $f\in\CFH$, then since $T^{-1}_\alpha=T_\alpha^*$ we have $$\scal{T_\alpha^{-1}f,g}_{\mathcal{H}_\alpha}=\scal{f,T_\alpha g}_{\CFH}\quad \forall g\in\mathcal{H}_\alpha.$$
Observe that $T_\alpha g$ is given by $$T_\alpha g(q)=\displaystyle\int_\R\mathcal{C}_\alpha(q,x)g(x)d\mu_\alpha(x).$$ Then, we have 
\begin{align*}
\displaystyle \scal{T_\alpha^{-1}f,g}_{\mathcal{H}_\alpha}
 &=\scal{f,T_\alpha g}_{\CFH}\quad \forall g\in\mathcal{H}_\alpha.
 \\& = \int_{\C_I}\overline{(T_\alpha g)^e_I(p)}f^e_I(p)d\lambda_{\alpha,I}(p)+2(\alpha+1)\int_{\C_I}\overline{(T_\alpha g)^o_I(p)}f^o_I(p)\vert{p}\vert^{-2}d\lambda_{\alpha+1,I}(p)
 \\
 \end{align*}
However, we have $$(T_\alpha g)^e_I(p)=\displaystyle\int_\R\mathcal{C}_\alpha^e(p,x)g(x)d\mu_\alpha(x)\text{ and } (T_\alpha g)^o_I(p)=\displaystyle\int_\R\mathcal{C}_\alpha^o(p,x)g(x)d\mu_\alpha(x).$$
 Thus, making use of Fubini's theorem we get that for any $g\in\mathcal{H}_\alpha$ we have
 \begin{align*}
\displaystyle \scal{T_\alpha^{-1}f,g}_{\mathcal{H}_\alpha}
 &=\int_\R\overline{g(x)}\psi_f(x)d\mu_\alpha(x)
 \\& =\scal{\psi_f,g}_{\mathcal{H}_\alpha},
 \\
 \end{align*}

where we have set $$\psi_f(x)=\displaystyle\int_{\C_I}\mathcal{C}_{\alpha,I}^e(\overline{p},x)f^I_e(p)d\lambda_{\alpha,I}(p)+2(\alpha+1)\int_{\C_I}\mathcal{C}^o_{\alpha,I}(\overline{p},x)f^I_o(p)\vert{p}\vert^{-2}d\lambda_{\alpha+1,I}(p).$$
Since the last equality holds for all $g\in\CFH$ it follows that $T^{-1}_\alpha f=\psi_f.$
\end{proof}
\begin{rem}
The integral on the right hand side in Proposition \ref{inv} does not depend on the choice of the imaginary unit since the scalar product does not depend on the choice of the slice.
\end{rem}
We finish this section by connecting this unitary transform $T_\alpha$ to what we call the (right) slice Dunkl transform. Indeed, we define the (right) slice Dunkl transform of a function $\varphi\in \mathcal{H}_\alpha$ with respect to a slice $\C_I$ to be $$\displaystyle D_\alpha^I\varphi(x):=\int_\R L_\alpha(-Ixt)\varphi(t)d\mu_\alpha(t)$$
More properties of the classical complex Dunkl transform can be found for example in \cite{Soltani2015}. Actually, this transform generalizes the classical Fourier transform on the real line. It satisfies a version of the Plancherel theorem since it extends uniquely to a unitary operator from the Hilbert space $\mathcal{H}_\alpha$ onto itself.
Then, we prove
\begin{lem}\label{Dunkl}
Let $I$ be any imaginary unit in $\mathbb{S}$ and $\psi\in\mathcal{H}_\alpha$. Then,  $$T_\alpha D_\alpha^I\psi(x)=T_\alpha\psi\circ g(x)\quad \text{where}\quad g(x)=-xI \quad \text{for all } x\in\R.$$
\end{lem}
\begin{proof}
Let $x\in\R$ and $\psi\in\mathcal{H}_\alpha$, and consider the function $$\varphi(s)=D_\alpha^I\psi(s):=\int_\R L_\alpha(-Ist)\psi(t)d\mu_\alpha(t)$$
Then, thanks to the Plancherel and Fubini's theorems $\varphi\in\mathcal{H}_\alpha$ we can write
\begin{align*}\displaystyle
T_\alpha\varphi(x)
 &= \int_\R \mathcal{C}_\alpha(x,s)\varphi(s)d\mu_\alpha(s)
 \\& =\int_\R \mathcal{C}_\alpha(x,s)\left(\int_\R L_\alpha(-Ist)\psi(t)d\mu_\alpha(t)\right) d\mu_\alpha(s)
 \\& = \int_\R \left(\int_\R\mathcal{C}_\alpha(x,s)L_\alpha(-Ist)d\mu_\alpha(s)\right)\psi(t) d\mu_\alpha(t)
 \end{align*}
 Note that, we have $$\displaystyle \phi(x,t):=\int_\R\mathcal{C}_\alpha(x,s)L_\alpha(-Ist)d\mu_\alpha(s).$$ Thus, we get $$\displaystyle \phi(x,t):=2^{\frac{\alpha+1}{2}}e^{-\frac{x^2}{2}}\int_\R e^{-\frac{s^2}{2}}L_\alpha(\sqrt{2}xs)L_\alpha(-Ist)d\mu_\alpha(s).$$ The last integral can be evaluated as in Theorem 3.4 in \cite{Soltani2015} since $x\in\R$. Then, we get $$\displaystyle \phi(x,t)=\mathcal{C}_\alpha(-Ix,t).$$
 Therefore, we obtain $$T_\alpha D_\alpha^I\psi(x)=\int_\R \mathcal{C}_\alpha(-Ix,t)\psi(t) d\mu_\alpha(t).$$
 Finally, this shows that $$T_\alpha D_\alpha^I\psi(x)=T_\alpha\psi(-xI).$$
\end{proof}
As a consequence of the Lemma \ref{Dunkl} we have the following :
\begin{prop}
For any $I\in\mathbb{S}$ and $f\in\CFH$ we have $$T_\alpha D^I_\alpha T^{-1}_\alpha(f)(x)=f(-xI) \quad \forall x\in\R.$$
\end{prop}
\begin{proof}
We just have to take $\psi=T^{-1}_\alpha f\in\mathcal{H}_\alpha$ and then apply the Lemma \ref{Dunkl}.
\end{proof}
\section{Some quaternionic operators on $\CFH$}

In this section, we shall consider the two following  operators on the slice Cholewinski-Fock space defined by $$\displaystyle \MS f(q)=qf(q)\quad \text{and} \quad \DS f(q)=\partial_Sf(q)+(2\alpha+1)q^{-1}f^o(q)$$ with domains given respectively by $$D(\MS)=\lbrace{f\in\CFH; \MS f\in\CFH}\rbrace$$   and  $$D(\DS)=\lbrace{f\in\CFH; \DS f\in\CFH}\rbrace.$$
Note that $\MS$ and $\DS$ are quaternionic right linear operators densely defined on $\CFH$ since $\lbrace{\phi_n^\alpha}\rbrace_{n\in\N}$ is an orthonormal basis of this quaternionic Hilbert space. \\ \\ In the sequel, we present some properties of these right quaternionic operators on $\CFH$
\begin{prop}
$\MS$ and $\DS$ are two closed quaternionic operators on $\CFH$.
\end{prop}

\begin{proof}
We consider the graph of $\MS$ defined by $$\mathcal{G}(\MS):=\lbrace{(f,\MS f); f\in D(\MS)}\rbrace.$$  Let us show that $\mathcal{G}(\MS)$ is closed. Indeed, let $\phi_n$ be a sequence in $D(\MS)$ such that $\phi_n$ and $\MS\phi_n$ converge to $\phi$ and $\psi$ respectively on $\CFH$. Then, thanks to the Proposition \ref{estimation} we have $$\vert{\phi_n(q)-\phi(q)}\vert\leq C_q \Vert{\phi_n-\phi}\Vert\quad \text{and} \quad \vert{\MS\phi_n(q)-\psi(q)}\vert\leq C_q \Vert{\MS\phi_n-\psi}\Vert; $$
it follows that $\phi_n$ and $\MS\phi_n$ converge respectively to $\phi$ and $\psi$ pointwise. This leads to $\psi(q)=\MS\phi(q)$ which ends the proof. The same technique could be adopted to prove the closeness of the operator $\DS$ on $\CFH$.
\end{proof}

\begin{prop}
Let $f\in\CFH$. Then, $\MS f\in\CFH$ if and only if $\DS f\in\CFH$. In particular this means that $D(\MS)=D(\DS)$.
\end{prop}
\begin{proof}
Let $\displaystyle f(q)=\sum_{n=0}^\infty q^na_n$ be an entire slice regular function belonging to $\CFH$; we shall compute $\Vert{\MS f}\Vert$ and $\Vert{\DS f}\Vert$. We have $$\displaystyle\MS f(q)=\sum_{n=0}^\infty q^{n+1}a_{n} \quad \text{and} \quad \Vert{\MS f}\Vert^2=\sum_{n=0}^\infty\beta_{n+1}(\alpha)\vert{a_{n}}\vert^2.$$ On the other hand,
\begin{align*}\displaystyle
\DS f(q)&= \sum_{k=1}^\infty 2k q^{2k-1}a_{2k}+\sum_{k=1}^\infty 2(\alpha+k+1)q^{2k}a_{2k+1}
 \\& =\sum_{k=1}^\infty \frac{\beta_{2k}(\alpha)}{\beta_{2k-1}(\alpha)}q^{2k-1}a_{2k}+\sum_{k=1}^\infty \frac{\beta_{2k+1}(\alpha)}{\beta_{2k}(\alpha)}q^{2k}a_{2k+1}
 \\& = \sum_{n=1}^\infty \frac{\beta_{n}(\alpha)}{\beta_{n-1}(\alpha)}q^{n-1}a_{n}
 \end{align*}
Thus we have  $$ \DS f(q)= \sum_{n=0}^\infty q^nc_n\quad \text{where} \quad c_n=\frac{\beta_{n+1}(\alpha)}{\beta_{n}(\alpha)}a_{n+1}.$$

Hence, making use of Proposition \ref{scal} we obtain $$\displaystyle\Vert{\DS f}\Vert^2=\sum_{n=0}^\infty \frac{\beta_{n+1}(\alpha)}{\beta_n(\alpha)}\beta_{n+1}(\alpha)\vert{a_{n+1}}\vert^2.$$
Now, we use the fact that $$\displaystyle \frac{\beta_{n+1}(\alpha)}{\beta_n(\alpha)}=n+1+\frac{2\alpha+1}{2}(1+(-1)^n)$$ and setting $k=n+1$ we get
$$\displaystyle\Vert{\DS f}\Vert^2=\sum_{k=0}^\infty \left(k+\frac{2\alpha+1}{2}(1-(-1)^k)\right)\beta_k(\alpha)\vert{a_k}\vert^2$$
This leads to $$\displaystyle\Vert{\DS f}\Vert^2=\Vert{\MS f}\Vert^2-\Vert{f}\Vert^2-(2\alpha+1)\sum_{k=0}^\infty(-1)^k\beta_k(\alpha)\vert{a_k}\vert^2$$
 Last equality concludes the proof and shows that $\MS$ and $\DS$ have the same domain on $\CFH$.
\end{proof}
\begin{prop}
For $f\in D(\DS)$ and $g\in D(\MS),$ we have $$\scal{\DS f,g}_{\CFH}=\scal{f,\MS  g}_{\CFH}$$
and $$\scal{\MS g,f}_{\CFH}=\scal{g,\DS f}_{\CFH}.$$
\end{prop}
\begin{proof}
Take $\displaystyle f(q)=\sum_{n=0}^\infty q^na_n$ and $\displaystyle g(q)=\sum_{n=0}^\infty q^nb_n$. Then, as we have seen before we have $$ \DS f(q)= \sum_{n=0}^\infty q^nc_n\quad \text{with} \quad c_n=\frac{\beta_{n+1}(\alpha)}{\beta_{n}(\alpha)}a_{n+1}$$ and by taking $b_{-1}=0$ we have $$\displaystyle\MS g(q)=\sum_{n=0}^\infty q^nb_{n-1}$$
Therefore, it follows from the Proposition \ref{scal} that
\begin{align*}\displaystyle
\scal{\DS f,g}_{\CFH}&= \sum_{n=0}^\infty \beta_{n+1}(\alpha)\overline{b_n}a_{n+1}
 \\& =\scal{f,\MS g}_{\CFH}.
 \\&
 \end{align*}
Then, we just need to apply $\overline{\scal{h,l}}=\scal{l,h}$ to get the second formula.
\end{proof}

\begin{prop}
The commutator of the operators $\DS$ and $\MS$ satisfies $$\left[\DS;\MS\right]=\Io+(2\alpha+1)A $$

where $\Io$ is the identity operator and $Af(q)=f(-q)$ on $\CFH$.
\end{prop}
\begin{proof}
Let $f\in\CFH$, then we have $$\displaystyle\MS f(q)=qf(q)\quad \text{and} \quad \DS f(q)=\partial_Sf(q)+(2\alpha+1)q^{-1}\left(\frac{f(q)-f(-q)}{2}\right).$$
Thus, $$\displaystyle\MS\DS f(q)=q\partial_S f(q)+(2\alpha+1)\left(\frac{f(q)-f(-q)}{2}\right)$$ Moreover, since the identity is an intrinsic entire slice regular function then the slice derivative satisfies the Leibniz formula so that we have $$\displaystyle\DS\MS f(q)=f(q)+q\partial_S f(q)+(2\alpha+1)\left(\frac{f(q)+f(-q)}{2}\right).$$
Hence, by substituting the two last equations we get the desired result.
\end{proof}
Finally, all the previous properties could be summarized in the following main result
\begin{thm}
$\MS$ and $\DS$ are closed densely defined right quaternionic linear operators adjoints of each other on the slice Cholewinski-Fock space. Moreover, they satisfy the commutation rule
$$\left[\DS;\MS\right]= \Io+(2\alpha+1)A.$$
\end{thm}
\begin{rem}
If $\displaystyle\alpha=-\frac{1}{2}$ the last theorem states that the slice derivative $\partial_S$ and the quaternionic multiplication operator $M_q$ are adjoints one of each other and satisfy the classical commutation rule $\left[\partial_S;M_q\right]=\Io$ on the slice hyperholomorphic Fock space introduced in \cite{AlpayColomboSabadini2014}.
\end{rem}
\section{The slice monogenic Cholewinski-Fock spaces}
Let $\lbrace{e_1,e_2, . . . , e_n}\rbrace$ be an orthonormal basis of the Euclidean vector space $\R^n$ with a non-commutative product defined by the following multiplication law
$$e_ke_s + e_se_k = -2\delta_{k,s};\quad k, s= 1, . . . , n$$
where $\delta_{k,s}$ is the Kronecker symbol. The set $\lbrace{e_A : A \subset\lbrace{1, . . . , m}\rbrace}\rbrace$ with
$e_A = e_{h1}e_{h2}...e_{hr}, 1 \leq h1 < ...< h_r \leq n, e_{\emptyset} = 1$
forms a basis of the Clifford algebra $\R_n$. Let $\R^{n+1}$ be embedded in $\R_n$ by identifying
$(x_0, x_1,..., x_n) \in \R^{n+1}$ with the para-vector $x=x_0+\underset{-}x\in \R_n$. The conjugate of $x$ is given by $\bar{x} = x_0-\underset{-}x$
and the norm of $x$ is defined by $\vert{x}\vert^2=x_0^2+...+x_n^2$. Furthermore, the $(n-1)$ dimensional sphere of units $1-$vectors in $\R^n$ is denoted by $$\mathbb{S}^{n-1}=\lbrace{x=x_1e_1+...+x_ne_n ; x_1^2+...x_n^2=1}\rbrace.$$
Note that if $I\in\Sq^{n-1}$, then $I^2=-1$. Based on these notations, in \cite{CSS2009} the theory of slice regular functions on quaternions was extended to the slice monogenic setting thanks to the following :
\begin{defn}
A real differentiable function $f: \Omega\subset\R^{n+1} \longrightarrow \R_n$ on a given open set is said to be a slice (left) monogenic function if, for very $I\in \Sq^n$, the restriction $f_I$ to the slice $\C_{I}$, with variable $x=u+Iv$, satisfies the following equation on $\Omega_I=\Omega\cap\C_I$

$$
\overline{\partial_I} f(u+Iv):=
\dfrac{1}{2}\left(\frac{\partial }{\partial u}+I\frac{\partial }{\partial v}\right)f_I(u+vI)=0.
$$
The space of all slice monogenic functions on $\Omega$ is denoted by $\mathcal{SM}(\Omega)$.
\end{defn} By analogy with the quaternionic setting, to $f\in \mathcal{SM}(\R^{n+1})$ such that $$f(x)=f^e(x)+f^o(x)$$
we consider $$\Vert{f}\Vert^2_{\alpha,n}:=\displaystyle \int_{\C_I}\vert{f^e_I(x)}\vert^2 d \lambda_{\alpha,I}(x)+2(\alpha+1)\int_{\C_I}\vert{f^o_I(x)}\vert^2\vert{x}\vert^{-2}d\lambda_{\alpha+1,I}(x)$$
where for the para-vector $x=u+vI\in\C_I$ we have $$\displaystyle d\lambda_{\alpha,I}(x):=\frac{\vert{x}\vert^{2\alpha+2}}{\pi 2^{\alpha}\Gamma(\alpha+1)}K_\alpha(\vert{x}\vert^2)d\lambda_I(x).$$
Hence, we define the slice monogenic Cholewinski-Fock space on $\R^{n+1}$ to be $$\mathcal{F}_{Slice}^\alpha(\R^{n+1}):=\lbrace{f\in\mathcal{SM}(\R^{n+1}); \Vert{f}\Vert_{n,\alpha}<\infty}\rbrace.$$
The monomials given by paravectors $({x^n})_n$ form an orthogonal basis of $\mathcal{F}_{Slice}^\alpha(\R^{n+1})$. Moreover, note that if $\displaystyle\alpha=-\frac{1}{2}$ this space is the slice  
hyperholomorhic Clifford Fock space introduced in section 5 of \cite{AlpayColomboSabadini2014}.\\ \\
Finally, we conclude this paper by the following
\begin{rem}
 The theory of slice monogenic functions with Clifford valued functions \cite{ColomboSabadiniStruppa2011,CSS2009} extends following the same spirit the one of slice regular functions on quaternions so that we have the same extended versions of : Splitting Lemma, series expansion theorem, Representation Formula, etc. Hence, most of the results obtained in this paper in the quaternionic setting could be rewritten in the slice monogenic setting.
\end{rem}
\noindent{\bf Acknowledgements} \\ \\
I would like to thank Prof. Irene Sabadini for reading an earlier version of this paper and for her interesting comments and suggestions on the topic. Special thanks to Prof. Daniel Alpay and Prof. Fabrizio Colombo for useful discussions.  Thanks to the anonymous referee for their important remarks that helped to improve the paper. The author acknowleges the support of the project “INdAM Doctoral Programme in Mathematics and/or Applications Cofunded by Marie Sklodowska-Curie Actions”, acronym: INdAM-DP-COFUND-2015, grant number: 713485.

\end{document}